\documentclass[11pt,letterpaper]{article}
\pdfoutput=1 

\usepackage{amsmath, amssymb, amsthm, amsfonts}
\usepackage[font=small,labelfont=bf]{caption}
\usepackage{graphicx}
\usepackage[outdir=./]{epstopdf}
\usepackage{xcolor}
\usepackage{booktabs}
\usepackage{algorithm}
\usepackage{algpseudocode}

\usepackage[xindy,acronyms]{glossaries}
\glsdisablehyper  

\usepackage{csvsimple}  
\usepackage{tikz}
\usetikzlibrary{calc}
\usetikzlibrary{positioning}
\definecolor{main}{RGB}{234, 234, 234}
\definecolor{secondary}{RGB}{4, 30, 62}
\definecolor{splitColor}{RGB}{202, 200, 200}
\definecolor{classColor1}{RGB}{163, 31, 52}
\definecolor{classColor2}{RGB}{4, 30, 65}
\definecolor{classColor3}{RGB}{0, 105, 143}
\definecolor{classColor4}{RGB}{4, 30, 65}
\tikzset{%
  every neuron/.style={
    circle,
    draw,
    minimum size=.5cm
  },
  neuron missing/.style={
    draw=none,
    scale=4,
    text height=0.333cm,
    execute at begin node=\color{black}$\vdots$
  },
}
\tikzset{class1/.style={draw,circle,text=white, fill=classColor1!80}}
\tikzset{class2/.style={draw,circle,text=white, fill=classColor2!80}}
\tikzset{class3/.style={draw,circle,text=white, fill=classColor3!80}}
\tikzset{class4/.style={draw,circle,text=white, fill=classColor4!80}}
\tikzset{split/.style={draw,rectangle,fill=splitColor!20}}
\tikzset{level 1/.style={sibling distance=12em},
         level 2/.style={sibling distance=6em},
         level 3/.style={sibling distance=3em}
        }

\usetikzlibrary{arrows}

\usepackage{hyperref}
\hypersetup{colorlinks=true,
            urlcolor=blue,
            linkcolor=black,
            citecolor=blue,
            bookmarksdepth=paragraph,
            pdftitle={Incremental cutting-plane method and its application},
            pdfauthor={Nagisa Sugishita},}

\usepackage{pgfplots, pgfplotstable}
\usepgfplotslibrary{statistics}
\usepgfplotslibrary{groupplots}
\usepgfplotslibrary{dateplot}
\usepgfplotslibrary{colorbrewer}

\usepackage[
  style=numeric,
  backend=biber,
  maxcitenames=2,
]{biblatex}

\theoremstyle{plain}
\newtheorem{theorem}{Theorem}[section]
\theoremstyle{plain}

\theoremstyle{plain}
\newtheorem{lemma}[theorem]{Lemma}
\theoremstyle{plain}
\newtheorem{proposition}[theorem]{Proposition}
\theoremstyle{definition}

\theoremstyle{definition}
\newtheorem{assumption}[theorem]{Assumption}
\theoremstyle{definition}

\theoremstyle{definition}
\newtheorem{algo}[theorem]{Algorithm}
\theoremstyle{remark}

\numberwithin{equation}{section}

\newcommand{\iprod}[2]{#1^T #2}

\newcommand{\citet}[1]{\textcite{#1}}

\usepackage{lipsum}

\newcommand*\samethanks[1][\value{footnote}]{\footnotemark[#1]}

\title{Incremental cutting-plane method and its application}
\author{Nagisa Sugishita\thanks{School of Mathematics, University of Edinburgh, James Clerk Maxwell Building, Edinburgh (UK), EH9 3FD, Email: {\tt n.sugishita@sms.ed.ac.uk}, {\tt a.grothey@ed.ac.uk}, {\tt k.mckinnon@ed.ac.uk}}, Andreas Grothey\samethanks, Ken McKinnon\samethanks}

\begin{document}

\maketitle

\begin{abstract}
We consider regularized cutting-plane methods to minimize a convex function that is the sum of a large number of component functions.
One important example is the dual problem obtained from Lagrangian relaxation on a decomposable problem.
In this paper, we focus on an incremental variant of the regularized cutting-plane methods, which only evaluates a subset of the component functions in each iteration.
We first consider a limited-memory setup where the method deletes cuts after a finite number of iterations.
The convergence properties of the limited-memory methods are studied under various conditions on regularization.
We then provide numerical experiments where the incremental method is applied to the dual problems derived from large-scale unit commitment problems.
In many settings, the incremental method is able to find a solution of high precision in a shorter time than the non-incremental method.

\end{abstract}

\section{Introduction}\label{sec1}

The focus of this paper is a convex optimization problem whose objective function is the sum of a large number of convex component functions.
One notable example is the dual problem derived from a decomposable optimization problem.

Cutting-plane methods are popular solution methods to solve such a convex optimization problem \cite{Kelley1960,CheneyAndGoldstein1959}.
Given the values and subgradients of the component functions evaluated at the previous points, the methods construct a piecewise affine model which minorizes the objective function.
The minimizer of the model is then used as the next point.
Namely, for each component function, its value and subgradient are evaluated at this new point and the methods compute an affine function which minorizes the component.
This affine function is referred to as a cut and used to update the aforementioned piecewise affine model.
The modified model is then minimized and the above procedure is repeated.
The plain cutting-plane method suffers instability and often regularization is added to the model to mitigate the issue \cite{Briantetal2008}.

In this paper, we consider an incremental variant of the regularized cutting-plane methods in a sense that it only evaluates a subset of the components in each iteration.
The study on the incremental variant of the cutting-plane methods is motivated by the success of other incremental first-order methods such as incremental gradient descent methods and incremental subgradient methods \cite{Nedic2001,Bertsekas2012}.
One of the most notable examples is the training of neural networks with stochastic gradient methods \cite{BertsekasAndTsitsiklis1996}, which is closely related to incremental gradient descent methods.
Empirically it has been observed that incremental methods often converge much faster than non-incremental methods, especially when the iterate is far from the solution.
It is also expected to be advantageous for the cutting-plane method to evaluate only a subset of the components and update the models and the iterate more frequently if the number of components is large.

In a typical approach, the piecewise affine model described above grows unlimitedly as the algorithm proceeds.
That is, the model uses data obtained in all of the previous iterations and its size increases every time a new point is obtained.
When the model gets large, it is of interest to delete some data (i.e.\ delete some cuts) and keep the model size moderate.
Heuristics may be used to select cuts to be deleted, but they are not supported by any convergence results \cite[Remark 9.8]{Bonnansetal2006}.
We study the behaviour of the regularized cutting-plane method in the limited-memory setup under various conditions on the regularization.

Note that similar ideas have been explored in the context of the bundle methods \cite{HiriartUrrutyAndLemarechal1993b}.
The bundle methods also construct a piecewise affine model of the objective function and solve the model with regularization.
Typically the bundle methods support the limited-memory setting thanks to the technique called cut aggregation.
Recently, extensions of the methods to an incremental setup have been studied \cite{VanAckooijetal2018b,EmielAndSagastizabal2010}.
However, the current approaches often require stronger assumptions (e.g.,\ knowledge on the Lipschitz constance of the objective) and further development is to be seen.
See \cite{VanAckooijetal2018b} for further discussion on recent advances in the incremental bundle methods.

The rest of the paper is structured as follows.
Section \ref{sec_incremental_cp_on_convex_problems} introduces the incremental cutting-plane method and studies its convergence property under a limited-memory assumption.
In section \ref{sec_numerical_experiments} we present some computational results.
Finally, in Section \ref{sec_conclusion} conclusions of this paper are discussed.

\section{Cutting-plane methods}
\label{sec_incremental_cp_on_convex_problems}

In this section, we study the following convex optimization problem:
\begin{align}
\min \ & f(x) = \sum_{i = 1}^m f_i(x)
\label{eq_cutting_plane_method_on_convex_problems_problem} \\
\text{s.t.} \ & x \in X, \notag
\end{align}
where we assume the following.
\begin{assumption}
\label{assumption_cp_on_convex_setup}
\
\begin{enumerate}
\item $X \subset \mathbb{R}^n$ is a nonempty, closed and convex set.
\item $f_i: \mathbb{R}^n \rightarrow \mathbb{R}$ is a convex function and is finite and subdifferentiable over $X$ for $i = 1, 2, \ldots, m$.
\item There exists $C > 0$ for which
\[
\|g_i(x)\| \le C, \qquad \forall i, x \in X, g_i(x) \in \partial f_i(x),
\]
where $\partial f_i(x)$ is the subdifferential of $f_i$ at $x$ for each $i$.
\end{enumerate}
\end{assumption}
This is a standard assumption to study incremental subgradient methods \cite{Nedic2001}.
We note that for each $i$, given the convexity of $f_i$, the boundedness of subgradients is equivalent to the $C$-Lipschitz continuity of $f_i$ \cite[Lemma 2.6]{ShalevShwartz2012}.
We write
$f^* = \inf_{x \in X} f(x) \in \mathbb{R} \cup \{-\infty\}$ and $X^* = \{x \in X \mid f(x) = f^*\}$, which may be empty.

\subsection{Algorithm}
\label{sec_limited_memory_cuting_plane_methods}

The cutting-plane method is an iterative solution method which computes candidate points $\{ x_k \}$ to solve \eqref{eq_cutting_plane_method_on_convex_problems_problem} approximately.
In the $k$th iteration, given the previous points $x_0, x_1, \ldots, x_k$, the method computes the next point $x_{k + 1}$ by minimizing a regularized model
\begin{align*}
x_{k + 1}
=
\arg\min \sum_{i = 1}^m \hat{f}_{k, i}(x) + \frac{1}{2 t_k} \| x - x_k \|^2,
\end{align*}
where $\hat{f}_{k, i}$ is a piecewise affine model of the $i$th component
\begin{align}
\hat{f}_{k, i}(x)
= \max_{l = 0, 1, \ldots, k} f_i(x_{k - l}) + \iprod{g_i(x_{k - l})}{(x - x_{k - l})},
\label{eq_model}
\end{align}
$f_i(x_l)$ and $g_i(x_l)$ are the value and any subgradient of $f_i$ at $x_l$ ($l = 0, 1, \ldots, k$) and $t_k$ is a parameter to adjust the strength of the regularization.
We refer to $t_k$ as a step size in a reason described later.
The values and subgradients of the components are evaluated at the new point $x_{k + 1}$ to update the model \eqref{eq_model} and the above process is repeated.

In this paper, we consider an incremental variant of the method in a sense that it only evaluates a subset of the components $f_i$ in each iteration.
We denote the index set of components evaluated in the $k$th iteration by $I_k$.
Since there are $m$ subproblems,
\[
I_k \subset \{1, 2, \ldots, m\}, \qquad \forall k.
\]
We note that the standard cutting-plane method is a special case where $I_k = \{1, 2, \ldots, m\}$ for all $k$.
We refer to this non-incremental method as the full-step method.

Furthermore, it is of interest to bound the memory requirement of the method.
To this end, we consider a method which only uses information obtained in the last $W$ iterations, where $W$ is a prescribed number.
That is, the piecewise affine models only use cuts obtained within the last $W$ iterations.

The algorithm is formally described below.

\begin{algo}
\label{algo_incremental_cp_on_convex}
\
\begin{enumerate}
\item Select $x_0 \in X$, memory size $W \in \mathbb{N}$, step size $\{t_k\}$, evaluation schedule $\{I_k\}$.
\item For $k = 0, 1, \ldots$, repeat the following:
\begin{enumerate}
\item Evaluate $f_i(x_k)$ and $g_i(x_k) \in \partial f_i(x_k)$ for $i \in I_k$.
\item Let
\begin{equation}
x_{k+1} = \arg\min {\breve f}_{k}(x),
\label{eq_algo_iteration}
\end{equation}
where
\begin{align*}
{\breve f}_{k}(x) & = \sum_{i = 1}^m {\hat f}_{k, i}(x) + \frac{1}{2 t_k} \| x - x_k \|^2, \\
{\hat f}_{k, i}(x) & =
\left\{\begin{array}{ll}
\max_{l \in B_{k, i}} f_i(x_{k-l}) + \iprod{g_i(x_{k-l})}{(x - x_{k-l})}, & \text{if } B_{k, i} \not= \emptyset \\
0, & \text{otherwise}
\end{array}\right. \\
B_{k, i} & = \{l \in \mathbb{Z} \mid 0 \le l \le \min\{k, W-1\}, i \in I_{k - l}\}.
\end{align*}
\end{enumerate}
\end{enumerate}
\end{algo}

Note that the above algorithm is well-defined:
For each $k$, $\breve{f}_k$ is a proper closed $(1/2t_k)$-strongly convex function and a minimizer of $\breve{f}_k$ exists and is unique \cite[Theorem 5.25]{Beck2017}.

We assume the following on some analysis below.

\begin{assumption}
\label{assumption_cp_on_convex_algorithm}
\
\begin{enumerate}
\item All components are evaluated at least once in every $W$ iterations:
\[
i \in \bigcup_{l=0}^{W-1} I_{k + l}, \qquad \forall k, i.
\]
\end{enumerate}
\end{assumption}

Under this assumption, each model has at least one cut after $W$ iterations (that is, $B_{k, i} \not= \emptyset$ for any $k \ge W - 1$ and $i$).

\subsection{Convergence analysis}

In this section, we study the behaviour of the algorithm.

The first lemma allows us to write the next point $x_{k + 1}$ in terms of the current point $x_k$ and the subgradients at the previous points.

\begin{lemma}
\label{lemma_cp_on_convex_projection}
Let Assumption \ref{assumption_cp_on_convex_setup} hold.
For $k=0, 1, \ldots$, there exists $\alpha_{k} \in \mathbb{R}^{m \times W}_+$ such that
\[
x_{k+1} = P_X(u_k),
\]
\[
u_k = x_k - t_k \sum_{i = 1}^m \sum_{l \in B_{k, i}} \alpha_{k, i, l} \, g_i(x_{k-l}),
\]
\[
\sum_{l \in B_{k, i}} \alpha_{k, i, l} =
\left\{
\begin{array}{ll}
1 & \qquad \text{if } B_{k, i} \not= \emptyset, \\
0 & \qquad \text{otherwise}, \\
\end{array}
\right.
\quad \forall i,
\]
\[
\alpha_{k, i, l} \ge 0, \quad \forall i, l,
\]
where $P_X$ denotes the (Euclidean) projection onto the set $X$ and we adopt the convention $\sum_\emptyset = 0$.
\end{lemma}

We observe that the iteration resembles that of the projected (non-incremental) subgradient methods.
In our method, for each component, we compute a convex combination of the subgradients evaluated in the last $W$ iterations (the weight is determined via an optimization problem in the proof below), sum over all of the components, multiply the ``step size'' $t_k$ and make a ``subgradient step'' followed by the projection.

\begin{proof}
Let $J_k$ be the index set of the components with at least one cut
\[
J_k = \{i \in \mathbb{Z}^+ \mid B_{k, i} \not= \emptyset\}.
\]
The minimization problem of ${\breve f}_k(x)$ can be written as
\begin{align*}
\min \ & \sum_{i \in J_k} r_i + \frac{1}{2t_k}\|x - x_k\|^2 \\
\text{s.t.} \
& r_i \ge f_i(x_{k-l}) + \iprod{g_i(x_{k-l})}{(x - x_{k-l})}, \quad \forall i \in J_k, \ l \in B_{k, i} \\
& x \in X, r \in \mathbb{R}^{\lvert J_k \rvert}.
\end{align*}
This optimization problem satisfies the Slater condition \cite[Proposition 5.3.1]{Bertsekas2009} and thus the (unique) optimal solution $(x^*, r^*)$ satisfies
\begin{equation*}
(x^*, r^*) = \arg\min_{x \in X, r \in \mathbb{R}^{\lvert J_k \rvert}} L(x, r, \alpha)
\end{equation*}
where
\begin{align*}
L(x, r, \alpha)
= \sum_{i \in J_k} r_i + \frac{1}{2t_k} \|x-x_k\|^2
- \sum_{i \in J_k} \sum_{l \in B_{k, i}} \alpha_{i, l} \left(r_i - f_i(x_{k-l}) - \iprod{g_i(x_{k-l})}{(x - x_{k-l})}\right)
\end{align*}
is the Lagrangian and $\alpha \ge 0$ is an optimal dual solution.
The optimality condition of $L$ gives
\begin{align*}
0 \in -\partial_{x} L(x^*, r^*, \alpha) + N_{X}(x^*), \quad
0 \in -\partial_{r} L(x^*, r^*, \alpha),
\end{align*}
where $N_{X}(x)$ is the normal cone to $X$ at $x$.
From the second relation, we obtain
\begin{align*}
\sum_{l \in B_{k, i}} \alpha_{i, l}
= 1 \qquad \forall i \in J_k
\end{align*}
while the first relation reads
\begin{align*}
0 \in -\frac{1}{t_k} (x^*-x_k) - \sum_{i \in J_k}\sum_{l \in B_{k, i}}\alpha_{i, l} g_i(x_{k-l}) + N_X(x^*).
\end{align*}
This implies
\begin{align*}
x^* = P_X\left(
  x_{k} - t_k \sum_{i \in J_k} \sum_{l \in B_{k, i}}\alpha_{i, l} g_i(x_{k-l})
\right).
\end{align*}
By relabelling $\alpha$, we obtain the desired relations.
\end{proof}

With the preceding lemma, one can estimate the distance between iterates.

\begin{lemma}
\label{lemma_cp_on_convex_one_step_bound}
Let Assumption \ref{assumption_cp_on_convex_setup} hold.
For any $k$ and $s$ with $k>s$, we have
\[
\|x_k - x_s\| \le \overline{t}_{s, k-1} (k-s)mC,
\]
where $\overline{t}_{s, k-1} = \max \{ t_s, t_{s + 1}, \ldots, t_{k - 1}\}$.
\end{lemma}

\begin{proof}
Using Lemma \ref{lemma_cp_on_convex_projection} and the nonexpansion property of the projection, for all $k \ge 0$, it follows
\begin{align*}
\|x_{k+1} - x_k\|
& = \|P_X(u_k) - x_k\| \\
& \le \|u_k - x_k\| \\
& = \left\|t_k \sum_{i=1}^{m} \sum_{l \in B_{k, i}} \alpha_{k, i, l} g_i(x_{k-l})\right\| \\
& \le t_k \sum_{i=1}^{m} \sum_{l \in B_{k, i}} \alpha_{k, i, l} \|g_i(x_{k-l})\| \\
& \le t_k \sum_{i=1}^{m} \sum_{l \in B_{k, i}} \alpha_{k, i, l} C \\
& \le t_kmC,
\end{align*}
where we have used $\alpha_{k, i, l} \le 1$ in the last inequality.
Combining the preceding inequality with the triangle inequality, we obtain the desired inequality
\begin{align*}
\|x_k - x_s\|
& = \left\|\sum_{l = 0}^{k - s - 1} (x_{s + l + 1} - x_{s + l})\right\| \\
& \le \sum_{l = 0}^{k - s - 1} \|x_{s + l + 1} - x_{s + l}\| \\
& \le \sum_{l = 0}^{k - s - 1} t_{s + l}mC \\
& \le \overline{t}_{s, k - 1} (k - s) mC.
\end{align*}
\end{proof}

Now we prove a fundamental inequality which plays a key role in the convergence analysis of the method.

\begin{lemma}
\label{lemma_cp_on_convex_fundamental_inequality}
Let Assumption \ref{assumption_cp_on_convex_setup} hold.
For any $y \in X$ and any $k \ge 0$ such that $B_{k, i} \not= \emptyset$ for $i = 1, 2, \ldots, m$, we have
\begin{align*}
\|x_{k+1} - y\|^2
\le \|x_k - y\|^2 - 2 t_k (f(x_k)-f(y)) + 4 t_k \overline{t}_{k-W, k-1}m^2C^2W + t_k^2 m^2 C^2,
\end{align*}
where $\overline{t}_{k-W, k-1} = \max\{t_{k-W}, t_{k-W+1}, \ldots, t_{k-1}\}$ and we define $t_i = t_0$ for $i < 0$ for notational convenience.
\end{lemma}

The condition $B_{k, i} \not= \emptyset$ for all $i$ is satisfied for any $k \ge W - 1$ if each component function is evaluated at least once in every $W$ iterations (that is, if Assumption \ref{assumption_cp_on_convex_algorithm} is satisfied).

\begin{proof}
Using Lemma \ref{lemma_cp_on_convex_projection} and the nonexpansion property of the projection, for all $k \ge 0$, it follows
\begin{align*}
& \|x_{k+1} - y\|^2 - \|x_{k} - y\|^2 \\
& \quad = \|P_X(u_k) - y\|^2 - \|x_{k} - y\|^2 \\
& \quad \le \|u_k - y\|^2 - \|x_{k} - y\|^2 \\
& \quad = \left\|x_k - t_k \sum_{i=1}^{m} \sum_{l \in B_{k, i}} \alpha_{k, i, l} g_i(x_{k-l}) - y\right\|^2 - \|x_{k} - y\|^2 \\
& \quad = - 2t_k \sum_{i=1}^{m} \sum_{l \in B_{k, i}} \alpha_{k, i, l} \iprod{g_i(x_{k-l})}{(x_k - y)} + t_k^2 \left\|\sum_{i=1}^{m} \sum_{l \in B_{k, i}}\alpha_{k, i, l} g_i(x_{k-l})\right\|^2.
\end{align*}

We shall bound the two terms in the last line.
Using the property of subgradient, for all $k \ge 0$, $i$ and $l \in B_{k, i}$, we obtain
\begin{align*}
-\iprod{g_i(x_{k-l})}{(x_k-y)}
& = -\iprod{g_i(x_{k-l})}{(x_{k-l}-y)} - \iprod{g_i(x_{k-l})}{(x_k - x_{k-l})} \\
& \le -(f_i(x_{k-l}) - f_i(y)) + \| g_i(x_{k-l}) \| \cdot \| x_k - x_{k-l} \| \\
& \le -(f_i(x_k) - f_i(y)) + (f_i(x_k) - f_i(x_{k-l})) + C \| x_k - x_{k-l} \|.
\end{align*}
The convexity of $f_i$ yields
\[
f_i(x_{k}) - f_i(x_{k-l})
\le \iprod{g_i(x_{k})}{(x_{k} - x_{k-l})}
\le \| g_i(x_{k}) \| \| x_{k} - x_{k-l} \|
\le C \| x_{k} - x_{k-l} \|
\]
where $g_i(x_k)$ is a subgradient of $f_i$ at $x_k$.
Therefore,
\begin{align*}
-\iprod{g_i(x_{k-l})}{(x_k-y)}
& \le -(f_i(x_k) - f_i(y)) + C\|x_{k-l} - x_k\| + C \| x_k - x_{k-l} \| \\
& \le -(f_i(x_k) - f_i(y)) + 2\overline{t}_{k-W, k-1}mC^2W,
\end{align*}
where the last inequality follows from Lemma \ref{lemma_cp_on_convex_one_step_bound}.
Thus, for any $k \ge 0$ such that $B_{k, i} \not= \emptyset$ for $i = 1, 2, \ldots, m$,
\begin{align*}
& - 2t_k \sum_{i=1}^{m} \sum_{l \in B_{k, i}} \alpha_{k, i, l} \iprod{g_i(x_{k-l})}{(x_k - y)} \\
& \quad \le 2t_k \sum_{i=1}^{m} \sum_{l \in B_{k, i}} \alpha_{k, i, l} \left(-(f_i(x_k)-f_i(y))+2\overline{t}_{k-W, k-1}mC^2W\right) \\
& \quad = - 2t_k (f(x_k)-f(y)) + 4t_k\overline{t}_{k-W, k-1}m^2C^2W,
\end{align*}
where the last equality holds since $\sum_{l \in B_{k, i}} \alpha_{k, i, l} = 1$ for all $i$ and $f(x) = \sum_{i = 1}^m f_i(x)$ for any $x \in X$.

Now we shall bound the other term.
With the triangle inequality, for all $k \ge 0$,
\begin{align*}
\left\|\sum_{i=1}^{m} \sum_{l \in B_{k, i}}\alpha_{k, i, l} g_i(x_{k-l})\right\|^2
& \le \left(\sum_{i=1}^{m} \sum_{l \in B_{k, i}}\alpha_{k, i, l} \|g_i(x_{k-l})\|\right)^2 \\
& \le \left(\sum_{i=1}^{m} \sum_{l \in B_{k, i}}\alpha_{k, i, l} C\right)^2 \\
& \le m^2 C^2.
\end{align*}
Combining the two bounds, we obtain the desired relation.
\end{proof}

The preceding lemma is an analogue of Lemma 2.1 from Nedic \cite{Nedic2001}, which is a basis of the convergence analysis of incremental subgradient methods.
Following their analysis of incremental subgradient methods, now we can obtain various convergence properties.
The three propositions below correspond to Proposition 2.1, 2.4 and 2.6 from Nedic \cite{Nedic2001}, respectively.
Proofs are similar to those in the reference and omitted.

\begin{proposition}
\label{prop_grouped_incremental_method_constant_step size}
Let Assumption \ref{assumption_cp_on_convex_setup} and \ref{assumption_cp_on_convex_algorithm} hold.
Assume that step size $\{ t_k \}$ is fixed to a positive constant $t$.
\begin{enumerate}
\item
If $f^* = -\infty$, then
\[
\liminf_{k \rightarrow \infty} f(x_k) = -\infty.
\]
\item
If $f^* > -\infty$, then
\[
\liminf_{k \rightarrow \infty} f(x_k) \le f^* + 2 m^2 t C^2W + \frac{1}{2} m^2 t C^2.
\]
\end{enumerate}
\end{proposition}

The above lemma estimates the worst-case suboptimality the method may suffer when the step size is kept constant.
It is counter-intuitive that the worst-case bound gets larger as the memory size $W$ increases.
We speculate that tightening the analysis in Lemma \ref{lemma_cp_on_convex_fundamental_inequality} would yield a better bound, but it is left as future work.

With a diminishing step size rule, the method is capable of solving the problem asymptotically.

\begin{proposition}
\label{prop_grouped_incremental_method_diminishing_step size_liminf}
Let Assumption \ref{assumption_cp_on_convex_setup} and \ref{assumption_cp_on_convex_algorithm} hold.
Assume that step size $\{t_k\}$ is such that
\[
t_k > 0 \ (\forall k), \ t_{k} \ge t_{k+1} \ (\forall k), \ \lim_{k \rightarrow \infty} t_k = 0, \ \sum_{k=0}^\infty t_k = \infty.
\]
Then, we have
\[
\liminf_{k \rightarrow \infty} f(x_k) = f^*.
\]
\end{proposition}

\begin{proposition}
\label{prop_grouped_incremental_method_diminishing_step size_convergence}
Let Assumption \ref{assumption_cp_on_convex_setup} and \ref{assumption_cp_on_convex_algorithm} hold.
Assume that step size $\{t_k\}$ is such that
\[
t_k > 0 \ (\forall k), \ t_{k} \ge t_{k+1} \ (\forall k), \ \sum_{k=0}^\infty t_k = \infty, \ \sum_{k=0}^\infty t_k^2 < \infty,
\]
and assume that the set of optimal solutions $X^*$ is nonempty.
Then, the sequence $\{x_k\}$ converges to some optimal solution.
\end{proposition}

\section{Numerical experiments}
\label{sec_numerical_experiments}

In this section, the incremental cutting-plane method is applied to the dual problems of unit commitment (UC) problems and the performance is compared with that of the full-step method.
The UC problem is an important optimization problem in the energy industry to compute the optimal operating schedules of power plants given demand over a fixed period.
The formulation of the problem is given in the appendix.
One of the standard approaches is to apply the Lagrangian relaxation to decompose the problem by generators and solve the dual problem with the cutting-plane method \cite{VanAckooijetal2018a}.

We consider UC problems with 200 generators whose data is based on \cite{Borghettietal2002}.
The length of the planning horizon is 48 hours with a time resolution of 1 hour.
The demand data is based on the historical demand data in the UK published by National Grid ESO.\footnote{https://www.nationalgrideso.com/}
We sample 8 different demand data and create test instances based on these demand data.

The incremental cutting-plane method used in this section is the one described in Algorithm \ref{algo_incremental_cp_on_convex}.
For simplicity, we consider the constant step size rule ($t_k = t$ for all $k$) and we do not delete any cuts ($W = \infty$).
When the method is full-step, which evaluates all the component values in each iteration, the objective value at each point is readily available.
In practice, this can be used to monitor the progress.
In our numerical experiments, we only update the iterate $x_k$ when the objective value gets improved.
Namely, if the objective value of the point computed by the model is worse than that of the current point, only the model gets updated and the iterate (i.e.\ the regularization centre) is kept to the same point.
This modification improves the performance of the full-step method and is often used in practice (e.g. \cite{Schulzeetal2017}).

The evaluation schedule $\{I_k\}$ is computed based on permutation.
In our experiments, the number of evaluated components per iteration, denoted by $p$, is fixed to be the same in all iterations.
Namely, $\lvert I_k \rvert = p$ for all $k$.
We consider a sequence of component functions where every $m$ elements are a permutation of $f_1, f_2, \ldots, f_m$
\begin{equation*}
\underbrace{f_{h_1^1}, f_{h_2^1}, \ldots, f_{h_m^1}}_{\text{1st permutation}},
\underbrace{f_{h_1^2}, f_{h_2^2}, \ldots, f_{h_m^2}}_{\text{2nd permutation}},
\ldots.
\end{equation*}
In every iteration, the first $p$ components are removed from the sequence and evaluated.
This loosely follows the idea considered in the study on incremental subgradient methods \cite{Nedic2001}, which corresponds to $p = 1$ where a single component is evaluated per iteration.
Setting $p = m$ we obtain the full-step method which evaluates all of the components in each step.

For later use, we introduce an adjusted step size $\widetilde{t} = mt / p$ where $t$ is the actual step size.
We note that the method with a small value of $p$ makes more steps (i.e.\ solves model \eqref{eq_algo_iteration} more often) than the method with a large value of $p$.
If all of the component functions are affine, the method makes the same progress per component evaluation no matter which $p$ value is used, as long as the adjusted step size $\widetilde{t}$ is the same.
Thus, the adjusted step size $\widetilde{t}$ facilitates the comparison of experiments with different values of $p$.

Two approaches to initialize the cutting-plane method are considered: warmstart and coldstart.
Coldstart is used when there is no prior knowledge of the problem to be solved.
However, if available, warmstart may significantly reduce the computational time.
It is of interest to verify if the incremental update helps in this practical setup or not.
In our numerical experiments, to warmstart the method we first solve the linear programming relaxation (LPR) of the original problem and use the optimal dual value to the LPR as the initial value of the cutting-plane method.
This initialization method was used in \cite{Schulzeetal2017} to warmstart the cutting-plane method.
In the coldstart case, the initial dual value is set to 0.

To observe the effect of the incremental update, the cutting-plane method is applied to the test instances with various numbers of component evaluations per iteration $\lvert I_k \rvert = p$.
Table \ref{tab_result_incremental_ratio} shows the average computational time, the average number of component evaluations and the average number of iterations to find a solution within 0.1\% and 0.05\% tolerance.
All times in this section are wall-clock times.
If the method is warmstarted, the computational time includes the time spent to solve the LPR, which is approximately 5 seconds.
For each value of $p$, the adjusted step size $\widetilde{t}$ is set to $0.01$ (i.e.\ the step size is $t = 0.01 p / m$).
The incremental method does not compute the function value $f(x_k)$ for $k = 0, 1, \ldots$ and thus they are evaluated after the experiments.

\begin{table}[htbp]
\begin{center}
\small
\caption{Performance of the cutting-plane method with various evaluation schedule sizes $p$}
\label{tab_result_incremental_ratio}
\begin{tabular*}{13.0cm}{
@{\extracolsep{\fill}}llrrrrrr}
\toprule
    &      & \multicolumn{3}{l}{tol: 0.1\%} & \multicolumn{3}{l}{0.05\%} \\
\cline{3-5}
\cline{6-8}
initializer & $p / m$ & time & comp. eval. & iter. & time & comp. eval. & iter. \\
\midrule
warmstart & 0.05 &          14.4 &   345.0 &   34.5 &          19.8 &   531.2 &   53.1 \\
    & 0.10 &  \textbf{13.8} &   347.5 &   17.4 &  \textbf{18.9} &   537.5 &   26.9 \\
    & 0.20 &          14.4 &   375.0 &    9.4 &          20.1 &   585.0 &   14.6 \\
    & 0.30 &          15.5 &   420.0 &    7.0 &          20.0 &   607.5 &   10.1 \\
    & 0.40 &          16.7 &   480.0 &    6.0 &          23.3 &   740.0 &    9.2 \\
    & 0.50 &          19.2 &   587.5 &    5.9 &          25.3 &   825.0 &    8.2 \\
    & 1.00 &          22.2 &   700.0 &    3.5 &          30.7 &  1025.0 &    5.1 \\
\midrule
coldstart & 0.05 &          23.5 &   911.2 &   91.1 &          39.0 &  1226.2 &  122.6 \\
    & 0.10 &          20.0 &   910.0 &   45.5 &          32.7 &  1232.5 &   61.6 \\
    & 0.20 &  \textbf{19.4} &   950.0 &   23.8 &  \textbf{29.2} &  1235.0 &   30.9 \\
    & 0.30 &          19.9 &  1005.0 &   16.8 &          29.5 &  1305.0 &   21.8 \\
    & 0.40 &          22.3 &  1110.0 &   13.9 &          31.8 &  1400.0 &   17.5 \\
    & 0.50 &          25.6 &  1237.5 &   12.4 &          36.1 &  1537.5 &   15.4 \\
    & 1.00 &          51.2 &  2075.0 &   10.4 &          66.9 &  2450.0 &   12.2 \\
\bottomrule
\end{tabular*}
\end{center}
\end{table}

As shown in Table \ref{tab_result_incremental_ratio}, for any value of $p$, the warmstarted method solves the problem quicker than the coldstarted method.
Namely, the LPR gives a dual value close to the optimal one, which helps the method to find a solution in a shorter time.

Interestingly, whichever initialization method is used, the incremental method ($p / m < 1$) performs better than the full-step method ($p / m = 1$).
In general, an incremental first-order method is expected to be more efficient than the full-step counterpart when the iterate is far from the solution, but the convergence to the solution of high precision is slower \cite{Bertsekas2012}.
The incremental cutting-plane method, however, successfully finds a solution of high precision (e.g.\ 0.05\%) in a short time.
We note that as the number of component evaluations per iteration $p$ gets smaller, the method tends to find a solution with fewer component evaluations in total.
However, at the same time, the number of iterations and hence the number of solutions of model \eqref{eq_algo_iteration} gets larger.
If $p$ is too small, the time spent to solve \eqref{eq_algo_iteration} becomes significant and the overall computational time increases.
The best value of $p$ balances the number of component evaluations in total and the growth of model solution time.

Table \ref{tab_result_stepsize} shows the performances of the method with various step sizes.
As shown in the table, the performance of the method is subject to the choice of the step size.
When the step size is very small, the numbers of component evaluations required by the incremental method and the full-step method tend to be close.
With a small step size, the regularization in the model \eqref{eq_algo_iteration} becomes strong and the methods cannot make much progress per iteration.
The behaviours of the two methods are similar in this situation.
However, in the setup in Table \ref{tab_result_stepsize}, the incremental method has to solve the model ten times more often than the full-step method per component evaluation.
As a result, it requires longer computational times.
On the other hand, an excessively large step size slows down the method as well.
Although this trend holds both for the incremental method and for the full-step method, the growth of the computational time is more significant in the full-step method.
In other words, the incremental method is more robust and tends to work relatively well even if the step size is set larger than the best value.

\begin{table}[htbp]
\small
\begin{center}
\caption{Performance of the cutting-plane method with various step sizes}
\label{tab_result_stepsize}
\begin{tabular*}{14.5cm}{
@{\extracolsep{\fill}}l@{\hskip0.1cm}c@{\hskip0.1cm}crrrrrr}
\toprule
     &     &       & \multicolumn{3}{l}{tol: 0.1\%} & \multicolumn{3}{l}{0.05\%} \\
\cline{4-6}
\cline{7-9}
initializer & $p / m$ & $\widetilde{t}$ & time &  comp. eval. & iter. &  time &  comp. eval. & iter. \\
\midrule
warmstart & 0.1 & 0.100 &          26.1 &   677.5 &   33.9 &          34.0 &   907.5 &   45.4 \\
          &     & 0.050 &          20.4 &   530.0 &   26.5 &          28.1 &   772.5 &   38.6 \\
          &     & 0.010 &          13.8 &   347.5 &   17.4 &  \textbf{18.9} &   537.5 &   26.9 \\
          &     & 0.005 &  \textbf{13.8} &   360.0 &   18.0 &          22.4 &   690.0 &   34.5 \\
          &     & 0.001 &          27.9 &   940.0 &   47.0 &          77.1 &  2735.0 &  136.8 \\
\midrule
          & 1.0 & 0.100 &          65.0 &  1775.0 &    8.9 &         101.1 &  2650.0 &   13.2 \\
          &     & 0.050 &          46.2 &  1350.0 &    6.8 &          66.0 &  1950.0 &    9.8 \\
          &     & 0.010 &          22.2 &   700.0 &    3.5 &          30.7 &  1025.0 &    5.1 \\
          &     & 0.005 &  \textbf{18.7} &   600.0 &    3.0 &  \textbf{27.7} &   975.0 &    4.9 \\
          &     & 0.001 &          26.6 &   950.0 &    4.8 &          70.5 &  2775.0 &   13.9 \\
\midrule
coldstart & 0.1 & 0.100 &          28.9 &   975.0 &   48.8 &          38.9 &  1195.0 &   59.8 \\
          &     & 0.050 &          24.4 &   900.0 &   45.0 &  \textbf{32.7} &  1095.0 &   54.8 \\
          &     & 0.010 &  \textbf{20.0} &   910.0 &   45.5 &          32.7 &  1232.5 &   61.6 \\
          &     & 0.005 &          25.2 &  1145.0 &   57.2 &          56.9 &  1902.5 &   95.1 \\
          &     & 0.001 &         143.3 &  4507.5 &  225.4 &         416.3 &  8847.5 &  442.4 \\
\midrule
          & 1.0 & 0.100 &         267.9 &  4300.0 &   21.5 &         347.9 &  4850.0 &   24.2 \\
          &     & 0.050 &         186.9 &  3750.0 &   18.8 &         218.8 &  4050.0 &   20.2 \\
          &     & 0.010 &          51.2 &  2075.0 &   10.4 &          66.9 &  2450.0 &   12.2 \\
          &     & 0.005 &  \textbf{35.3} &  1725.0 &    8.6 &  \textbf{56.5} &  2350.0 &   11.8 \\
          &     & 0.001 &          98.5 &  4500.0 &   22.5 &         235.2 &  8875.0 &   44.4 \\
\bottomrule
\end{tabular*}
\end{center}
\end{table}

Figure \ref{fig_obj_value} shows the objective values at the points computed by the method on one of the test instances.
The left and right figures correspond to the warmstart and the coldstart respectively.
In each plot, the full-step method and incremental method are shown in solid and dashed lines, respectively.
The warmstarted method first solves the LPR, which approximately take 5 seconds, while the coldstarted method starts outputting the iterates immediately.
In both cases, the full-step method with step size 0.1 often outputs points whose objective values are out of the range and do not appear in the figures.
%
%
\begin{figure}
    \begin{center}
    \small
\begin{tikzpicture}
  \begin{axis}[
      title={warmstart},
      xlabel={elapse (second)},
      ylabel={obj.\ value},
      xmin=0,
      xmax=20,
      ymin=0.990,
      ymax=1.002,
      restrict y to domain*=0.500:1.006,
      ytick={0.990, 0.994, 0.998, 1.002},
      yticklabels={0.990, 0.994, 0.998, 1.002},
      width=.45\textwidth,
      height=0.25\textheight,
  ]
  \tikzset{
    s0/.style={black},
    s1/.style={red},
    s2/.style={blue},
  };

  \edef\lstStepsizes{0.1,0.01,0.001}
  \pgfplotsinvokeforeach{0,1,2}{
    \pgfmathsetmacro{\stepsize}{{\lstStepsizes}[#1]}
    \edef\temp{\noexpand
    \addplot+ [s#1, mark=x] table [
            x=time,
            y=lb,
            col sep=comma
          ] {./data/plot_obj_vs_time_workspace_v4_200_lpr_instance_1700000_incremental_ratio_1.0_stepsize_\stepsize_version_v1.csv};
    };  
    \temp
    \edef\temp{\noexpand
    \addplot+ [s#1, mark=x, dashed, forget plot] table [
            x=time,
            y=lb,
            col sep=comma
          ] {./data/plot_obj_vs_time_workspace_v4_200_lpr_instance_1700000_incremental_ratio_0.1_stepsize_\stepsize_version_v1.csv};
    };  
    \temp
  }
  \end{axis}
\end{tikzpicture}
\hskip 0.05\textwidth
\begin{tikzpicture}
  \begin{axis}[
      title={coldstart},
      xlabel={elapse (second)},
      ylabel={obj.\ value},
      xmin=0,
      xmax=20,
      ymin=0.8,
      ymax=1.002,
      restrict y to domain*=0.2:1.006,
      legend pos=south east,
      width=.45\textwidth,
      height=0.25\textheight,
  ]
  \addlegendimage{empty legend}
  \tikzset{
    s0/.style={black},
    s1/.style={red},
    s2/.style={blue},
  };

  \edef\lstStepsizes{0.1,0.01,0.001}
  \pgfplotsinvokeforeach{0,1,2}{
    \pgfmathsetmacro{\stepsize}{{\lstStepsizes}[#1]}
    \edef\temp{\noexpand
    \addplot+ [s#1, mark=x] table [
            x=time,
            y=lb,
            col sep=comma
          ] {./data/plot_obj_vs_time_workspace_v4_200_coldstart_instance_1700000_incremental_ratio_1.0_stepsize_\stepsize_version_v1.csv};
    };  
    \temp
    \edef\temp{\noexpand
    \addplot+ [s#1, mark=x, densely dashed, forget plot] table [
            x=time,
            y=lb,
            col sep=comma
          ] {./data/plot_obj_vs_time_workspace_v4_200_coldstart_instance_1700000_incremental_ratio_0.1_stepsize_\stepsize_version_v1.csv};
    };  
    \temp
  }
  \legend{\hspace{-.7cm}stepsize,0.1,0.01,0.001};
  \end{axis}
\end{tikzpicture}
\end{center}
\caption{Objective values of the iterates by the warmstarted (left) and coldstarted (right) method.  Solid and dashed lines correspond to the full-step method and the incremental method respectively.}
\label{fig_obj_value}
\end{figure}
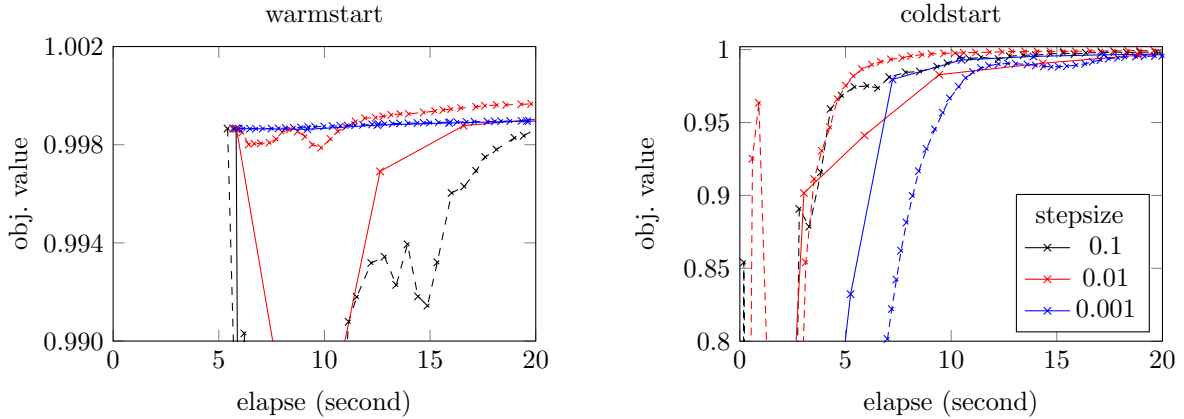

The figure on the left shows that the strength of the regularization (i.e.\ the reciprocal of the step size) has a significant effect on the behaviour of the full-step method, which is plotted in the solid lines.
We note that all of the cases use the same initial point which is close to the optimal solution, computed by the LPR.
However, the full-step method with the weakest regularization (step size 0.1) is quite unstable and it outputs dual values far away from the optimal one.
This is expected behaviour since the model does not have enough cuts in the first few iterations and without suitable regularization the model tends to output points far from the solution, even if it is initialized at a near-optimal point.
When the step size is 0.01, the full-step method still shows instability at first but makes constant progress afterwards.
If the step size decreased to the further smaller value 0.001, the full-step method does not show such unstable behaviour at all but the entire progress becomes significantly slow.
As a result, the full-step method with step size 0.001 takes longer than that of step size 0.01.

The same discussion applies to the warmstarted incremental method.
However, the instability of the incremental method observed with step size 0.1 and 0.01 is drastically improved compared with the full-step method.
The incremental method also shows the unstable behaviour at first but becomes stable more quickly than the full-step method, which leads to the reduction of the overall computational time.
However, the incremental update does not fix the slow progress caused by the too small step size 0.001 and the full-step method and incremental method behave quite similarly.

Similar trends are observed with the coldstart case which is plotted on the right.
We still observe the advantage of the incremental update with the large step sizes (0.1 and 0.01).
Namely the incremental update mitigate the instability of the cutting-plane method.
However, the incremental method with the small step size (0.001) is rather slower than the full-step method.

\section{Conclusion}
\label{sec_conclusion}

In this paper, we considered an incremental variant of regularized cutting-plane methods.
First, the convergence property of the method was studied under the limited-memory assumption.
That is, we considered a setting where cuts were deleted after a finite number of iterations.
By showing the similarity of the incremental cutting-plane method and incremental subgradient methods, the convergence of the method was established.
Then, the method was applied to the dual problems of large-scale unit commitment problems.
In many settings, the incremental method found a solution of high precision in a shorter time than the full-step method.
The incremental method mitigated the instability of the cutting-plane method observed in the first few iterations, which seemed to be one of the factors to account for the superior performance of the incremental method.

\printbibliography

\appendix

\section{Problem formulation}

We follow one of the standard formulations in literature and formulate the following constraints:
\begin{itemize}
\item \textbf{Load balance}: Generators have to meet the all demand in each time period.
\item \textbf{Reserve}: To deal with contingencies, it is required to keep a sufficient amount of back up in each time period, which can be activated quickly.
\item \textbf{Power output bounds}: Each generator's power output has to be within its limit.
\item \textbf{Ramp rate}: Generators can only change their outputs within the ramp rates.
\item \textbf{Minimum up/downtime}: If switched on (off), each generator has to stay on (off) for a given minimum period.
This is to avoid thermal stress in the generators which may cause wear and tear of the turbines.
\end{itemize}

To formulate the model, we use the following notation.

\begin{itemize}
\item{Sets}
\begin{itemize}
\item $G = \{1, 2, \ldots, n_G\}$: set of generators
\item $T = \{1, 2, \ldots, n_T\}$: set of time indices where decisions are taken
\end{itemize}

\item{Parameters}
\begin{itemize}
\item $C^{\mathrm{nl}}_{g}$: no-load cost of generator $g$
\item $C^{\mathrm{mr}}_{g}$: marginal cost of generator $g$
\item $C^{\mathrm{up}}_{g}$: startup cost of generator $g$
\item $P^{\max/\min}_{g}$: maximum/minimum generation limit of generator $g$
\item $P^{\mathrm{ru}/\mathrm{rd}}_{g}$: operating ramp up/down limits of generator $g$
\item $P^{\mathrm{su}/\mathrm{sd}}_{g}$: startup/shutdown ramp limits of generator $g$
\item $T^{\mathrm{u}/\mathrm{d}}_{g}$: minimum uptime/downtime of generator $g$
\item $P^{\mathrm{d}}_{t}$: power demand at time $t$
\item $P^{\mathrm{r}}_{t}$: reserve requirement at time $t$
\end{itemize}

\item{Variables}
\begin{itemize}
\item $\alpha_{gt} \in \{0, 1\}$: 1 if generator $g$ is on in period $t$, and 0 otherwise
\item $\gamma_{gt} \in \{0, 1\}$: 1 if generator $g$ starts up in period $t$, and 0 otherwise
\item $\eta_{gt} \in \{0, 1\}$: 1 if generator $g$ shut down in period $t$, and 0 otherwise
\item $p_{gt} \ge 0$: power output of generator $g$ in period $t$
\end{itemize}
\end{itemize}

The objective is the total cost
\[
\min \sum_{t \in T} \sum_{g \in G}
\left( C^{\mathrm{nl}}_g \alpha_{gt} + C^{\mathrm{mr}}_g p_{gt} +
C^{\mathrm{up}}_g \gamma_{gt}
\right).
\]
This is to be minimized subject to the following constraints.
\begin{itemize}
\item Load balance
\begin{equation*}
\sum_{g \in G} p_{gt} \ge P^{\mathrm{d}}_{t}
\qquad t \in T
\label{eq:uc_first_constraint}
\end{equation*}

\item Reserve
\begin{equation*}
\sum_{g \in G} (P^{\max}_{g} \alpha_{gt} - p_{gt})
\ge P^{\mathrm{r}}_t
\qquad t \in T
\end{equation*}

\item Power output bounds
\begin{equation*}
P^{\min}_{g} \alpha_{gt} \le p_{gt} \le P^{\max}_{g} \alpha_{gt}
\qquad g \in G, t \in T
\end{equation*}

\item Ramp rate
\begin{equation*}
p_{gt} - p_{g \, t-1} \le P^{\mathrm{ru}}_g \alpha_{g \, t-1}
+ P^{\mathrm{su}}_g \gamma_{gt}
\qquad g \in G, t \in T \backslash \{1\}
\end{equation*}
\begin{equation*}
p_{g \, t-1} - p_{gt} \le P^{\mathrm{rd}}_g \alpha_{gt}
+ P^{\mathrm{sd}}_g \eta_{gt}
\qquad g \in G, t \in T \backslash \{1\}
\end{equation*}

\item Minimum up/downtime
\begin{equation*}
\sum_{i=\max\{t-T^\mathrm{u}_g+1, 1\}}^t \gamma_{gi} \le \alpha_{gt}
\qquad g \in G, t \in T
\end{equation*}
\begin{equation*}
\sum_{i=\max\{t-T^\mathrm{u}_g+1, 1\}}^t \eta_{gi} \le 1 - \alpha_{gt}
\qquad g \in G, t \in T
\end{equation*}

\item Polyhedral/Switching constraints (to enforce binaries to work as we expect)
\begin{equation*}
\alpha_{gt} - \alpha_{g \, t-1} = \gamma_{gt} - \eta_{gt}
\qquad g \in G, t \in T
\end{equation*}
\begin{equation*}
1 \ge \gamma_{gt} + \eta_{gt}
\qquad g \in G, t \in T
\label{eq:uc_last_constraint}
\end{equation*}
\end{itemize}


\end{document}